\documentclass[11pt]{article} 
\usepackage{amssymb, latexsym}
\newtheorem{defin}{}
\newtheorem{saetze}[defin]{}
\newtheorem{conjec}[defin]{}
\newtheorem{lemmas}[defin]{}
\newtheorem{folger}[defin]{}
\newtheorem{bemerk}[defin]{}

\newenvironment{theorem}  {\begin{saetze}\it {\bf Theorem:}}{\end{saetze}}
\newenvironment{conj}     {\begin{conjec}\it {\bf Conjecture:}}{\end{conjec}}

\newenvironment{proof}    {{\it Proof}:}{{\hfill \fillbox \bigskip}}

\newcommand{\fillbox}{\mbox{$\bullet$}}

\newcommand{\bulit}{\item[$\bullet$]}
\newenvironment{items}{\begin{list}{$\alph{item})$}
{\labelwidth18pt \leftmargin18pt \topsep3pt \itemsep1pt \parsep0pt}}
{\end{list}}

\newcommand{\ra}{\rightarrow}

\newcommand{\ms}{\mapsto}

\newcommand{\N}{\mathbb N}
\newcommand{\F}{\mathbb F}

\newcommand{\ad}{\mathrm{ad}}

\begin{document}

\title{Graded Lie algebras of Cartan type \\ in characteristic 2}
\author{Tara Brough and Bettina Eick}
\date{\today}
\maketitle

\begin{abstract}
We investigate the graded Lie algebras of Cartan type $W, S$ and $H$ in
characteristic $2$ and determine their simple constituents and some 
exceptional isomorphisms between them. We also consider the graded Lie
algebras of Cartan type $K$ in characteristic $2$ and conjecture that
their simple constituents are isomorphic to Lie algebras
of type $H$.
\end{abstract}

%%%%%%%%%%%%%%%%%%%%%%%%%%%%%%%%%%%%%%%%%%%%%%%%%%%%%%%%%%%%%%%%%%%%%%%%%%%%%
\section{Introduction}

The simple Lie algebras over the complex numbers were first classified
by Killing (1888) and Cartan (1894). They fall in four infinite families
$A_n$, $B_n$, $C_n$ and $D_n$ and five exceptional cases $E_6, E_7, E_8,
G_2$ and $F_4$. All of these Lie algebras have analogues in the modular 
case and thus can be defined over arbitrary fields of characteristic $p$.
The resulting Lie algebras are called {\em classical} Lie algebras and 
they are simple for $p \geq 5$. We refer to \cite[Chapter 4]{Str04} for
background.

Over fields of characteristic $p$ there are also non-classical simple
Lie algebras known. There are the four infinite families of {\em graded
Cartan type}: the Witt algebras $W,$ the special algebras $S,$ the 
Hamiltonian algebras $H$ and the contact algebras $K$. The Lie algebras 
of {\em (general) Cartan type} are deformations of these. Additionally, for
$p = 5$ there exists the family of Melikian type Lie algebras. The 
classification of the simple Lie algebras over algebraically closed fields 
of characteristic $p \geq 5$ has been completed around the beginning of 
this century by Strade and Premet. It asserts that every simple Lie algebra 
over an algebraically closed field of characteristic $p$ is either 
classical or of Cartan or Melikian type. We refer to \cite{PSt06} and 
\cite{Str04} for background.

Over fields of characteristic $p$ with $p \in \{2,3\}$ the classification 
of simple Lie algebras is wide open. In particular for $p = 2$ it seems
that many new phenomena arise and the classification will differ significantly
from those in characteristic $0$ and $p \geq 5$. The classical Lie algebras
and the Lie algebras of Cartan type have analogues in characteristic $2$,
but these are not necessarily simple. The constituents of the classical Lie 
algebras in characteristic $2$ have been determined independently by 
Hogeweij \cite{Hog82} and Hiss \cite{His84}. On the other hand various
new infinite series of simple Lie algebras in characteristic $2$ have
been discovered, see for example \cite{Kap82}, \cite{Lin92}, \cite{Lin93}, 
\cite{Lin88} or \cite{Bro95}. 

In the main part of this note we consider the graded Lie algebras of
Cartan types $W, S$ and $H$ in characteristic $2$. We determine their
simple constituents and exhibit some exceptional isomorphisms between
them. As a result, we obtain that the following list of simple Lie algebras
in characteristic $2$ contains the simple constituents of all graded Lie 
algebras of Cartan types $W, S$ and $H$ up to isomorphism. 

\begin{figure}[htb]
\begin{center}
\begin{tabular}{llll}
algebra & parameters & dimension & notes \\
\hline
$W(n, m)$ & $n > 1$ 
    & $n 2^{m_1 + \ldots + m_n}$ 
    & Theorem 2.4 in \cite[Chap 4]{SFa88} \\ 
$W(1, (l))'$ & $l > 1$ 
    & $2^{l-1}$
    & Theorem \ref{witt} below \\
$S(n, m)$ & $n > 2$ 
    & $(n-1) (2^{m_1 + \ldots m_n} - 1)$
    & Theorem 3.5 in \cite[Chap 4]{SFa88} \\
$S(2, m)'$ & $1 \not \in m$ 
    & $2^{m_1 + m_2} - 2$ 
    & Theorem \ref{special} below \\
$H(n, m)$ & $n > 3$ even 
    & $2^{m_1 + \ldots + m_n} - 2$ 
    & Theorem \ref{hamilton} below \\
\hline
\end{tabular}
\end{center}
\caption{Simple Lie algebras}
\end{figure}

There are further isomorphisms among the Witt algebras, the special algebras and
the Hamiltonian algebras known. For example, two Witt algebras $W(n, m)$ 
and $W(n, m')$ are isomorphic if $m$ is a permutation of $m'$. Moreover, 
in characteristic $2$ there are further isomorphisms possible. For example, 
computations based on the computer algebra system GAP \cite{Gap} show that 
$H(4, (1,1,1,1)) \cong S(3, (1,1,1))$ and $H(4, (2,1,1,1)) \cong 
S(3, (2,1,1))$. This induces the following conjecture.

\begin{conj}
$H(4, (m_1, \ldots, m_3, 1)) \cong S(3, (m_1, \ldots, m_3))$ in
characteristic $2$.
\end{conj}

In the final section of this note, we consider the graded Lie algebras of
contact type $K$. Based on experimental evidence, we conjecture that their 
simple constituents are isomorphic to quotients of Lie algebras of type $H$.

%%%%%%%%%%%%%%%%%%%%%%%%%%%%%%%%%%%%%%%%%%%%%%%%%%%%%%%%%%%%%%%%%%%%%%%%%%%%%
\section{Preliminaries}

We briefly recall the definition of the polynomial algebras $A(n)$ and
$A(n,m)$ for $n \in \N$ and $m \in \N^n$. Let $\F$ be a field of 
characteristic $p > 0$ and let $X_1, \ldots, X_n$ be $n$ pairwise 
commuting indeterminates over $\F$. For $a \in \N^n$ we write $X^a$ 
for $X_1^{a_1} \cdots X_n^{a_n}$. Let $A(n)$ denote the commutative 
algebra consisting of all formal sums over $\F$ of the form 
$\sum_{a \in \N^n} \alpha_a X^{a}$ equipped with the usual addition and 
the multiplication
\[ (\sum_a \alpha_a X^a) (\sum_b \beta_b X^b) 
     = \sum_c (\sum_{a+b=c} \alpha_a \beta_b {c \choose a} ) X^c, \]
where the multi-binomial coefficient is evaluated modulo $p$ and thus 
is considered as an element in the prime field of $\F$. For $m \in \N^n$ 
let $\tau = (p^{m_1}-1,\ldots,p^{m_n}-1)$. Then $A(n, m) = \langle X^a 
\mid 0 \leq a \leq \tau \rangle \leq A(n)$ is a subalgebra of dimension 
$p^{m_1 + \cdots m_n}$. Let $D_j$ denote the partial derivation in $X_j$ 
on $A(n,m)$. 

%%%%%%%%%%%%%%%%%%%%%%%%%%%%%%%%%%%%%%%%%%%%%%%%%%%%%%%%%%%%%%%%%%%%%%%%%%%%%
\section{The Witt algebras}

The Witt algebra $W(n,m)$ is defined as the set of elements 
$\{ \sum_{j=1}^n f_j D_j \mid f_j \in A(n,m) \}$ equipped with 
the usual addition and the Lie bracket
\[ [ f D_i, g D_j ] = fD_i(g)D_j - g D_j(f) D_i + fg[D_i,D_j]. \]
The set $\{X^{(a)}D_i \mid 1\leq i\leq n, 0\leq a\leq \tau\}$ is a basis
for $W(n,m)$ and $W(n,m)$ has dimension $n \dim(A(n,m))$.

The algebra $W(1,(1))$ has dimension 2 and thus is solvable. In Theorem 
2.4 of Chapter 4 in \cite{SFa88} is shown that $W(n,m)$ is simple if 
$n > 1$. We consider the remaining cases in the following theorem.

\begin{theorem} \label{witt}
Let $char(\F) = 2$ and $W = W(1, m)$ with $m \neq (1)$. Then 
$\dim(W') = \dim(W)-1$ and $W'$ is simple.
\end{theorem}

\begin{proof}
Let $m = (l)$ with $l \in \N, l \neq 1$. Let $a_i := X^{(i)}D_1$ for 
$0 \leq i \leq 2^l-1$. Then $\{a_i \mid 0 \leq i \leq 2^l-1\}$ is a 
basis for $W$. Let $c_{ijk}$ denote the corresponding structure 
constants. Then $c_{ijk} = 0$ if $k \neq i+j-1$ and otherwise
\[ c_{ijk} = {i+j-1\choose i} - {i+j-1\choose j}. \]
Hence $[a_0,a_i] = a_{i-1}$ for $0 \leq i \leq 2^l-1$ (with $a_{-1} := 0$). 
Thus $a_i \in W'$ for $0 \leq i \leq 2^l-2$ and the structure constants also 
imply that $a_{2^l-1} \not \in W'$. Thus $W'$ has codimension $1$ in $W$.

Suppose that $I$ is a non-zero ideal in $W'$. Let $x = \Sigma_{i=0}^{2^l-1} 
\lambda_i a_i$ be a non-zero element in $I$. Let $k$ be maximal with 
$\lambda_k \neq 0$.  Then
\[ (\mathrm{ad} a_0)^k (x) 
   = \Sigma_{i=0}^{k} \lambda_i a_{i-k} 
   = \lambda_k a_0.\]
Hence $a_0 = \lambda_k^{-1}(\mathrm{ad} a_0)^k (x) \in I$. As $[a_0,a_i]
= a_{i-1}$ for $0 \leq i \leq 2^l-1$, it follows that $W' = \langle a_0, 
\ldots 2^l-2 \rangle \leq I$ and thus $W' = I$. 
\end{proof}

%%%%%%%%%%%%%%%%%%%%%%%%%%%%%%%%%%%%%%%%%%%%%%%%%%%%%%%%%%%%%%%%%%%%%%%%%%%%%
\section{The special algebras}

Let $n > 1$ and $m \in \N^n$. We define
\[ div : W(n,m) \ra A(n,m) : \sum_{j=1}^n f_j D_j \ms \sum_{j=1}^n D_j(f_j). \]
Then the special Lie algebra $S(n, m)$ is the derived subalgebra of the 
kernel of $div$. Define $D_{ij} : A(n,m) \ra W(n,m) : f \ms D_j(f)D_i - 
D_i(f)D_j$. Then $S(n,m)$ is generated by $\{ D_{ij}(f) \mid 
f \in A(n,m), 0 \leq i < j \leq n\}$ and has the dimension 
$(n-1) (\dim(A(n,m))-1)$.

The algebra $S(2,(1,1))$ is solvable. The algebras $S(n,m)$ for $n > 2$
are simple as shown in \cite{SFa88}, Theorem 3.5 in Chapter 4. The 
remaining cases are considered in the following theorem. Note that 
$S(2,(m_1,m_2)) \cong S(2,(m_2,m_1))$ for every $m_1, m_2 \in \N$.

\begin{theorem} \label{special}
Let $char(\F) = 2$ and $S = S(2, m)$ with $m \neq (1,1)$.
\begin{items}
\item[\rm (a)] 
If $1 \not \in m$, then $\dim(S') = \dim(S) - 1$ and
$S'$ is simple.
\item[\rm (b)]
If $m = (1,m_2)$, then $S'/N(S) \cong W(1, (m_2))'$. 
\end{items}
\end{theorem}

\begin{proof}
We first investigate $S = S(2,m)$ before we consider (a) and (b). Let 
$x_j = X^{(0,j)}D_1$ and $y_i = X^{(i,0)}D_2$ and $z_{ij} = X^{(i+1,j)}D_1 - 
X^{(i, j+1)}D_2$. Then $\{x_j, y_i, z_{ij} \mid  0\leq i\leq 2^{m_1}-2, 
0\leq j\leq 2^{m_2}-2\}$ is a basis for $S$. Define $x_{-1}, y_{-1} = 0$ 
and 
\[ \gamma_{ijkl} = {i+k+1\choose i+1}\left({j+l\choose j} + 
   {j+l\choose j-1}\right) - {i+k+1\choose i}\left({j+l\choose j+1} 
   + {j+l\choose j}\right). \] 
Then 
\begin{items}
\bulit
$[x_i,x_j] = 0$ for $0\leq i,j\leq 2^{m_2}-2$;
\bulit
$[y_i,y_j] = 0$ for $0\leq i,j\leq 2^{m_1}-2$;
\bulit
$[z_{ij},z_{kl}] = \gamma_{ijkl}z_{i+k,j+l}$ for $0\leq i,k\leq 2^{m_1}-2$
and $0\leq j,l\leq 2^{m_2}-2$;
\bulit
$[x_0,y_i] = y_{i-1}$ for $0\leq i\leq 2^{m_1}-2$;
\bulit
$[y_0,x_j] = x_{j-1}$ for $0\leq j\leq 2^{m_2}-2$;
\bulit
$[x_i,y_j] = z_{j-1,i-1}$ for $0<i\leq 2^{m_2}-2$ and $0<j\leq 2^{m_2}-2$;
\bulit
$[x_i,z_{0k}] = {i+k+1\choose i} x_{i+k}$ for $0\leq i,k\leq 2^{m_2}-2$;
\bulit
$[x_i,z_{jk}] = {i+k+1\choose i} z_{j-1,i+k}$ for 
          $0\leq i,k\leq 2^{m_2}-2$, $0<j\leq 2^{m_1}-2$;
\bulit
$[y_i,z_{j0}] = {i+j+1\choose i} y_{i+j}$ for
          $0\leq i,j\leq 2^{m_1}-2$;
\bulit
$[y_i,z_{jk}] = {i+j+1\choose i}z_{i+j,k-1}$ for 
          $0\leq i,j\leq 2^{m_1}-2$, $0<k\leq 2^{m_2}-2$.
\end{items}

These calculations imply that $\langle x_j, y_i z_{lk} \mid 
0\leq i,l\leq 2^{m_1}-2, 0\leq j,k\leq 2^{m_2}-2,(l,k) \neq 
(2^{m_1}-2, 2^{m_2}-2) \rangle \leq S'$. Thus $S'$ has codimension
at least 1 in $S$. A detailed inspection of the structure constants 
of $S$ further shows that $w := z_{2^{m_1}-2, 2^{m_2}-2} \not \in S'$ 
and thus $S'$ has codimension exactly $1$ in $S$.

(a) We consider the case $1 \not \in m = (m_1, m_2)$. Our aim is to show 
that $S'$ is simple. The proof is very similar to that for the Witt algebra, 
except that $a_{2^l-1}$ there is replaced by $w = z_{2^{m_1}-2, 2^{m_2}-2}$ 
here.  Let $I$ be a non-zero ideal of $S'$. Using the above commutators, it 
is sufficient to show that $x_0 \in I$ to obtain that $I = S'$. Let $v = 
\sum_{i=0}^{2^{m_2}-2} \lambda_i x_i + \sum_{j=0}^{2^{m_2}-2} \mu_j y_j 
+ \sum_{i=0}^{2^{m_1}-2} \sum_{i=0}^{2^{m_2}-2} \nu_{ij} z_{ij}$ be a 
non-zero element in $I$. 

Suppose first that all coefficients $\nu_{ij}$ are zero. If $\lambda_i
\neq 0$ for some $i$, then let $k$ be the greatest such that $\lambda_k
\neq 0$. Then 
\[ (\mathrm{ad} y_0)^k (v) 
    = (\mathrm{ad} y_0)^k (\sum_{i=0}^{k} \lambda_i x_i 
                         + \sum_{j=0}^{2^{m_1}-2} \mu_j y_j)
    = \sum_{i=0}^{k} \lambda_i x_{i-k} = \lambda_k x_0, \]
hence $x_0 = \lambda_k^{-1} \mathrm{ad} y_0 (v) \in I$. 
If $\lambda_i = 0$ for all $i$, then let $k$ be the greatest such that 
$y_k\neq 0$. Then 
\[ \mathrm{ad} x_1 (\mathrm{ad} x_0)^k (v) 
    = \mathrm{ad} x_1 (\mu_k y_0) 
    = \mu_k x_0, \] 
hence $x_0 = \mu_k^{-1}\mathrm{ad} x_1 (\mathrm{ad} x_0)^k (v)\in I$.
This proves the claim if all coefficients $\nu_{ij}$ are zero. The
other case that there is a non-zero $\nu_{ij}$ can be proved by similar
techniques; we leave this to the reader.

(b) Now we consider the case that $m = (1,m_2)$ with $m_2 \neq 1$. In this
case $S'$ has the basis $\{x_0, \ldots, x_{2^{m_2}-2}, y_0, z_{00}, \ldots, 
z_{0, 2^{m_2}-3}\}$. We define $\phi: S' \ra W(1, (m_2))'$ by $\phi(x_i) = 0$
and $\phi(y_0) = a_0$ and $\phi(z_{0i}) = a_{i+1}$. It is technical, but not
difficult to verify that this is an epimorphism of Lie algebras whose kernel
$\langle x_i \mid 0 \leq i \leq 2^{m_2}-2 \rangle$ is an abelian ideal in
$S$. Hence $S'/N(S) \cong W(1,(m_2))'$ is simple.
\end{proof}

%%%%%%%%%%%%%%%%%%%%%%%%%%%%%%%%%%%%%%%%%%%%%%%%%%%%%%%%%%%%%%%%%%%%%%%%%%%%%
\section{The Hamiltonian algebra}

We assume that $n \geq 2$ is even and set $n = 2r$.  For $1 \leq j \leq r$ 
let $\sigma(j) = 1$ and $j' = j+r$. For $r < j \leq n$ let $\sigma(j) = -1$ 
and $j' = j-r$. Define
\[ D_H : A(n,m) \ra W(n,m) : f \ms \sum_{j=1}^n \sigma(j) D_j(f) D_{j'}.\]
Then the kernel of $D_H$ is $\F 1$ and $H(n,m)$ is defined as the derived
subalgebra of the image of $D_H$. The algebra $H(n,m)$ is generated by 
the images $\{D_H(X^a) \mid 0 \leq a < \tau\}$ and has the dimension 
$\dim(A(n,m)) - 2$.

$H(2, (1,1))$ is solvable. The other cases are considered in the following 
theorem.

\begin{theorem} \label{hamilton}
Let $char(\F) = 2$ and $H = H(n, m)$ with $n = 2r$ even and $m \in \N^n$.
\begin{items}
\item[\rm (a)]
If $n > 2$, then $H$ is simple.
\item[\rm (b)]
If $n = 2$, then $H \cong S(2,m)'$.
\end{items}
\end{theorem}

\begin{proof}
(a) Let $h_a:=D_H(X^{(a)})$. Suppose that $I$ is a non-zero ideal of $H$, 
and let $w = \sum_{0 < a <\tau} \lambda_a h_a$ be a non-zero element of 
$I$.  Let $\Lambda(w) = \{a \mid \lambda_a\neq 0\}$. Then $w = 
\sum_{a\in \Lambda(w)} \lambda_a h_a$. Also let $\mu_i = 
\mathrm{max}\{a_i \mid a\in \Lambda(w)\}$ for $1\leq i\leq n$. Then at 
least one $\mu_i$ is non-zero, since $w\neq 0$ and so some (non-zero) 
vector is in $\Lambda(w)$. Let $k$ be minimal with $\mu_k\neq 0$. Now 
define $c_0 = \mu_k$ and
\[ c_i = \mathrm{max} \{a_{k+i} \mid a\in \Lambda(w), a_j=c_j \;
                                (k\leq j < k+i)\}. \]
Then the vector $v:= \sum_{i=0}^{n-k} c_i \varepsilon_i$ is in $\Lambda(w)$.

Now
\[ \ad h_{\varepsilon_{i'}}^j(h_a) = 
        \sigma(i))^j h_{a-j\varepsilon_{i}}. \]  
In particular, if $a_i<j$, then    
$(\ad h_{\varepsilon_{i'}})^j (h_a) = 0$.  Thus
\[ (\prod_{i=1}^{n-k} (\sigma(i))^{c_i} 
   (\ad h_{\varepsilon_{(k+i)'}})^{c_i})(w)=\lambda_v h_{c_0\varepsilon_k}, \]
since all $a\in \Lambda(w)$ other than $v$ have at least one $a_{k+i}<c_i$ 
for some $1\leq i\leq n-k$, hence the only term which is not taken to zero 
by this action is $\lambda_v h_v$.

Now
\[ (\ad h_{k'})^{c_0-1} (h_{c_0\varepsilon_k}) = 
         \sigma(k') h_{\varepsilon_k}, \] 
and thus $h_{\varepsilon_k}\in I$.

Next we show that all $h_{\varepsilon_i}$ are in $I$. First,
\[ [h_{\varepsilon_k}, h_{\varepsilon_{i}+\varepsilon_{k'}}] 
       = \sigma(k) h_{\varepsilon_{i}}\in I, \]
thus all $h_{\varepsilon_{i}}$ with $1\leq i\leq n$ are in $I$ provided 
that $h_{\varepsilon_{i}+\varepsilon_{k'}}<\tau$. This holds for all 
$i\neq k'$, and is also true for $i=k'$ provided $m_{k'}>1$. Suppose 
$m_{k'}=1$. Then, since $n>2$, there exists $1\leq i\leq n$ with $i 
\notin \{k, k'\}$. For such $i$ it follows that $h_{\varepsilon_i}\in I$
and $h_{\varepsilon_{k'+i'}}<\tau$ and:
\[ [h_{\varepsilon_i}, h_{\varepsilon_{k'+i'}}] 
           = \sigma(i) h_{\varepsilon_{k'}}\in I. \]
Hence all $h_{\varepsilon_{i}}$, $1\leq i\leq n$, are in $I$.

Finally, we show for arbitrary $0<a<\tau$ that $h_a\in I$. First suppose 
that $a\neq \tau - \varepsilon_i$ for some $1\leq i\leq n$. Then there 
exists $1\leq j\leq n$ such that $0<a+\varepsilon_j<\tau$, and
\[ [h_{\varepsilon_{j'}}, h_{a+\varepsilon_j}] 
     = \sigma(j') h_a\in I. \]  
For $1\leq i\leq n$ let $a= \tau - \varepsilon_i$. Note that 
$h_{\varepsilon_i+\varepsilon_i'}\in I$ since 
$h_{\varepsilon_i+\varepsilon_i'}\neq \tau - \varepsilon_j$ for any 
$1\leq j\leq n$. Now since
\[ [h_{\varepsilon_i+\varepsilon_{i'}}, h_a] 
      = \sigma(i) (a_{i'} - a_{i}) h_a
      = \sigma(i) (p^{m_i'} - 1 - (p^{m_i} - 2)) h_a 
      \equiv \sigma(i) h_a \;\mathrm{mod}\; p, \]
we see that $h_a\in I$.  

In summary, we obtain that $h_a\in I$ for all $0<a<\tau$, and thus we 
conclude that $I=H$. Hence $H$ is simple.
\medskip

(b) By construction, we observe that $D_H(X^{(a)}) = D_1(X^{(a)})D_2 
- D_2(X^{(a)})D_1$. 
Thus it follows that $\{D_H(X^{(a)}) \mid 0\leq a < \tau\} = \{D_{ij}(X^{(a)}) 
\mid 1\leq i < j\leq 2, 0\leq a \leq \tau\}$. The right hand side of this
equation is a basis for $S(2,m)'$ and the left hand side generates $H$.
\end{proof}

%%%%%%%%%%%%%%%%%%%%%%%%%%%%%%%%%%%%%%%%%%%%%%%%%%%%%%%%%%%%%%%%%%%%%%%%%%%%%
\section{The contact algebra}

We assume that $n \geq 3$ is odd and set $n = 2r+1$. For $f \in A(n,m)$ 
and $1 \leq j \leq n-1$ let $f_j = X^{\epsilon_j} D_n(f) + 
\sigma(j') D_{j'}(f)$. Further, let $f_n = 2f - \sum_{j=1}^{2r}
\sigma(j) X^{\epsilon_j} f_{j'}$. Then we define
\[ D_K : A(n,m) \ra W(n,m) : f \ms \sum_{j=1}^n f_j D_j. \]
Then $D_K$ is an injective linear map. We define $K(n,m)$ as the derived
subalgebra of the image of $D_K$. Hence $K(n,m)$ is generated by 
$\{D_K(X^a) \mid 0 \leq a \leq \tau \}$ and has dimension $\dim(A(n,m))$.

Based on computational evidence, we propose the following conjecture.

\begin{conj}
Let $char(\F) = 2$ and $K = K(n,m)$ for $n \in \N$ odd and $m \in \N^n$.
Then $K'/N(K')$ is simple and is isomorphic to $H/N(H)$, where $H = 
H(n-1, (m_1,\ldots,m_{n-1}))$.
\end{conj}

Note that $H/N(H) = H$ unless $n=3$ and $m_1=1$ (without loss of generality), 
in which case $H/N(H)\cong W(1,(m_2))'$. 

%%%%%%%%%%%%%%%%%%%%%%%%%%%%%%%%%%%%%%%%%%%%%%%%%%%%%%%%%%%%%%%%%%%%%%%%%%%%%
\section*{Acknowledgements}

This research was supported by a DAAD (German Academic Exchange Service) grant
under the program ``Research Grants for Doctoral Candidates and Young Academics 
and Scientists", which enabled the first author to visit the second author at 
TU Braunschweig from August 2006 - January 2007.

\def\cprime{$'$}

\end{document}